\documentclass[leqno,11pt,english]{amsart}

\usepackage{amssymb,amsmath,exscale,times,amsthm}

\usepackage{graphicx}
\usepackage{dsfont}
\usepackage[applemac]{inputenc}
\usepackage[T1]{fontenc}

\usepackage{color}
\definecolor{marin}{rgb}   {0.,   0.3,   0.7} 
\definecolor{rouge}{rgb}   {0.8,   0.,   0.} 
\definecolor{sepia}{rgb}   {0.8,   0.5,   0.} 

\usepackage{hyperref}

\theoremstyle{plain}
\newtheorem{theorem}{Theorem}[section]
\newtheorem{definition}[theorem]{Definition}

\newtheorem{lemma}[theorem]{Lemma}

\newtheorem{proposition}[theorem]{Proposition}
\theoremstyle{remark}
\newtheorem{remark}[theorem]{Remark}
\newtheorem{notation}[theorem]{Notation}
\newtheorem{example}[theorem]{Example}

\numberwithin{equation}{section}
\newcommand{\QED}{\mbox{}\hfill \raisebox{-0.2pt}{\rule{5.6pt}{6pt}\rule{0pt}{0pt}} 
          \medskip\par}

\newcommand{\dd}{\mathrm{d}}

\newcommand{\N}{\mathbb{N}}
\newcommand{\Nc}{\mathcal{N}}

\newcommand{\C}{\mathbb{C}}

\newcommand{\T}{\mathbb{T}}
\newcommand{\Z}{\mathbb{Z}}

\newcommand{\eps}{\varepsilon}
\newcommand{\ttt}{{\tt t}}
\newcommand{\tx}{{\tt x}}

\def\F{\mathcal F}
\def\O{\mathcal O}

\def\tv{{\tt v}}
\def\tu{{\tt u}}

\def\({\left(}
\def\){\right)}
\def\<{\left\langle}
\def\>{\right\rangle}
\def\le{\leqslant}
\def\ge{\geqslant}

\def\1{{\bf 1}}

\def\Tend#1#2{\mathop{\longrightarrow}\limits_{#1\rightarrow#2}}

\def\d{{\partial}}
\def\l{\lambda}
\def\om{\omega}

\def\g{\gamma}

\def\eps{\varepsilon}

\begin{document}

\title[Energy cascades for NLS]{Energy cascades for NLS on $\T^d$} 

\author[R. Carles]{R\'emi Carles}
\address{CNRS \& Univ. Montpellier~2\\Math\'ematiques \\
  CC~051\\F-34095 Montpellier} 
\email{Remi.Carles@math.cnrs.fr}
\author[E.~Faou]{Erwan Faou}
\address{INRIA \& ENS Cachan Bretagne\\
Avenue Robert Schumann\\ F-35170 Bruz\\ France} 
\email{Erwan.Faou@inria.fr}
\begin{abstract}
  We consider the nonlinear Schr\"odinger equation with cubic
  (focusing or defocusing) nonlinearity on the multidimensional
  torus. For special small initial data containing only five modes, we
  exhibit a countable set of time layers in which arbitrarily large
  modes are created. The proof relies on a reduction to multiphase
  weakly nonlinear geometric optics, and on the study of a particular
  two-dimensional discrete dynamical system.
\end{abstract}

\thanks{This work was supported by the French ANR project
  R.A.S. (ANR-08-JCJC-0124-01)}   
 \maketitle

\section{Introduction and main result}
\label{sec:intro}
We consider the nonlinear Schr\"odinger equation
\begin{equation}
  \label{eq:nls0}
  i\d_t u +\Delta u = \l |u|^2 u,\quad x\in \T^d,
\end{equation}
with $d\ge 2$, where the sign of $\l\in \{-1,+1\}$ turns out to be
irrelevant in the analysis below. In the present analysis, we are
interested in the description of some energy exchanges between low and
high frequencies for particular solutions of this equation. We will
consider solutions with small initial values:  
\begin{equation}
  \label{eq:ci0}
  u(0,x) = \delta u_0(x),
\end{equation}
where $u_0\in H^1(\T^d)$ and $0<\delta \ll 1$. Replacing $u$ with
$\delta^{-1}u$, \eqref{eq:nls0}--\eqref{eq:ci0} is equivalent to
\begin{equation}
  \label{eq:nls}
  i\d_t u +\Delta u = \l \eps |u|^2 u\quad ;\quad u(0,x)=u_0(x),
\end{equation}
where $\eps=\delta^2$. 
Viewed as an infinite dimensional dynamical system in terms of the
Fourier variables of the solution, such an equation is {\em resonant}
in the sense that all the eigenvalues of the Laplace operator are
integers only, making possible nontrivial vanishing linear
combinations between the frequencies of the linear unperturbed
equation ($\eps = 0$). In such a situation, the perturbation
theory cannot be directly applied as in
\cite{Bam03,BG06,Bou96,CrWa94,ElKuk,FG10,Greb07}. Let us recall that in all
these works, the Laplace operator is perturbed by a typical potential
making resonances {\em generically} disappear.  In such situations and
when $u_0(x)$ is smooth enough, it is possible to prove the quasi
preservation of the Sobolev norms of the solution over very long time:
polynomial (of order $\eps^{-r}$ for all $r$) as in \cite{BG06},
exponentially large as in \cite{FG10}, or arbitrary large for a set of
specific solutions as in \cite{ElKuk}.  

In the resonant case considered in this paper, there is {\em a priori}
no reason to observe long times bounds for the Sobolev norms of the
solution.  
Despite this fact, Eqn.~\eqref{eq:nls0} possesses many quasi-periodic
solutions (see \cite{Bou98,Wang}).  

On the other hand, it has been recently shown in \cite{Iturbulent}
that in the defocusing case (Eqn.~\eqref{eq:nls0} with $\lambda = 1$),
solutions exist exhibiting energy transfers between low and high modes
which in turn induce a growth in the Sobolev norm $H^s$ with $s >
1$. Strikingly, such  phenomenon arises despite the fact that  $L^2$
and $H^1$ norms of the solution are bounded for all time.  

The goal of the present work is to describe {\em quantitatively} such
energy exchanges in the case of a particular explicit initial value
$u_0(x)$ made of five low modes.  
 Since we work on $\T^d$, the solution $u$ takes the form
\begin{equation*}
  u(t,x) = \sum_{j\in \Z^d} u_j(t) e^{ij\cdot x},
\end{equation*}
where $u_j(t) \in \C$ are the Fourier coefficients of the solution, and 
as long as $t$ does not exceed the lifespan of $u$. Here, for $j =
(j_1,\ldots,j_d) \in \Z^d$ and $x = (x_1,\ldots,x_d)$, we have $j\cdot
x = j_1 x_1 + \cdots + j_d x_d$. We also set $|j|^2 = j_1^2 + \cdots +
j_d^2 \in \N$. Let us introduce the Wiener algebra $W$ made of functions $f$ on
$\T^d$ of the form 
\begin{equation*}
  f(x) = \sum_{j\in \Z^d}b_j e^{ij\cdot x}
\end{equation*}
such that $(b_j)_{j\in \Z^d}\in \ell^1(\Z^d)$. With this space is
associated the norm  
\begin{equation*}
  \|f\|_W = \sum_{j\in \Z^d}|b_j|. 
\end{equation*}

Our main result is the following: 

\smallbreak

\begin{theorem}\label{theo:main}
 Let $d\ge 2$, and $u_0\in C^\infty(\T^d)$ given by
 \begin{equation*}
   u_0(x) = 1+ 2\cos x_1 +2 \cos x_2. 
 \end{equation*}
For $\l\in \{\pm 1\}$, the following holds. There exist
 $\eps_0,T>0$ such that for $0<\eps\le \eps_0$, \eqref{eq:nls} has a
 unique solution $u\in
 C([0,T/\eps];W)$, and there exist $C_0>0$ and $C>0$ such that:
\begin{equation*}
\forall j\in \Nc_*, \, \exists c_j \neq 0, \, \forall t \in [0,T/\eps],  \quad
    \left\lvert u_j\(t\) - c_j (\eps t)^{|j|^2 - 1}\right  \rvert \le
    (C_0 \eps t )^{|j|^2} + C\eps, 
\end{equation*}
where the set $\Nc_*$ is given by 
\begin{equation}
\label{Netoil}
  \Nc_*=\left\{ (0,\pm 2^p),(\pm 2^p,0),
(\pm 2^p,\pm 2^p),(\mp 2^p,\pm
    2^p),\ p\in \N\right\}\times\{0_{\Z^{d-2}}\}.
\end{equation}
Arbitrarily high modes appear with equal intensity along a cascade of
time layers:
 \begin{align*}
   & \forall \g \in ]0,1[, \ \forall \theta<\frac{1}{4},\ \forall
   \alpha>0,\quad \exists 
   \eps_1\in ]0, \eps_0],  \quad \forall
   \eps\in ]0,\eps_1],\\
   & \forall j\in \Nc_*,\ \lvert j\rvert <\alpha \(\log
   \frac{1}{\eps}\)^\theta ,\quad \left\lvert
     u_j\(\frac{2}{\eps^{1-\g/(|j|^2-1)}}\)\right\rvert \ge
   \frac{\eps^\g}{4} .
\end{align*}
\end{theorem}

\smallbreak
This result expresses the possibility of nonlinear exchanges in
\eqref{eq:nls}: while the high modes in the set $\Nc_*$ are equal to
zero at time $t = 0$, they are significantly large in a time that
depends on the mode. As this time increases with the size of the mode,
this is an {\em energy cascade} in the sense of \cite{cricri}.  To our 
knowledge, this result is the first one where such a 
dynamics is described so precisely as to quantify the time of ignition
of different modes. 
\smallbreak


The proof of this theorem relies on the following ingredients: 

\begin{itemize}
\item An approximation result showing that the analysis of the
  dynamics of \eqref{eq:nls0} over a time of order
  $\mathcal{O}(1/\varepsilon)$ can be reduced to the study of  an
  infinite dimensional system for the amplitudes of the Fourier
  coefficients $u_j$ (in a geometric optics framework). Let us mention
  that this reduced system exactly corresponds to the {\em resonant
    normal form} system obtained after a first order Birkhoff
  reduction (see \cite{KP} for the one dimensional case). We detail
  this connection between geometric optics and normal forms in the
  second section.  
\item A careful study of the dynamics of the reduced system. Here
  we use the particular structure of the initial value which consists
  of five modes generating infinitely many new frequencies through the
  resonances interactions in the reduced system. A Taylor expansion
  (in the spirit of \cite{Car07})
  then shows how all the frequencies should be {\em a priori} turned
  on in finite time. The particular geometry of the energy repartition
  between the frequencies then makes possible to estimate precisely
  the evolution of the particular points of the set $\Nc_*$ in
  \eqref{Netoil} and to quantify the energy exchanges between them.  
\end{itemize}

The construction above is very different from the one in
\cite{Iturbulent}. Let us mention that it is also valid only up to a
time of order $\mathcal{O}(1/\varepsilon)$ (which explains the
absence of difference between the focusing and defocusing
cases). After this time, the
nature of the dynamics should change completely as all the frequencies
of the solution would be significantly  present in the
system, and the nature of the nonlinearity should become relevant. 

\section{An approximation result in geometric optics}
\label{sec:og}


For a given element $(\alpha_j)_{j\in \Z^d}\in \ell^1(\Z^d)$, 
we define the following infinite dimensional {\em resonant system} 
\begin{equation}
  \label{eq:a}
  i\dot a_j = \l \sum_{(k,\ell,m)\in I_j}a_k\overline a_\ell a_m \quad
  ;\quad a_j(0)=\alpha_j. 
\end{equation}
where   $I_j$ is the set of \emph{resonant} indices (see \cite{KP})
associated with $j$ defined by: 
  \begin{equation}
    \label{eq:Ij}
    I_j =\left\{(k,\ell,m)\in \Z^{3d}\mid \, j = k  - \ell +m, 
\mbox{ and }  |j|^2 = |k|^2 - |\ell|^2+|m|^2\, \right\}.  
  \end{equation}
With these notations, we have the following result: 
\begin{proposition}
\label{prop:approx}
Let $u_0(x) \in W$ and $(\alpha_j)_{j\in \Z^d}\in \ell^1(\Z^d)$ its
Fourier coefficients. There exists $T>0$ and 
  a unique analytic solution $(a_j)_{j\in \Z^d}: [0,T] \to
  \ell^1(\Z^d)$ to the system \eqref{eq:a}.  
Moreover, 
there exists $\eps_0(T)>0$ such that for
$0<\eps\le \eps_0(T)$, the exact solution  to 
\eqref{eq:nls} satisfies $u \in C([0,\frac{T}{\varepsilon}];W)$ and
there exists $C$ 
independent of $\eps\in ]0,\eps_0(T)]$ such that
\begin{equation*}
  \sup_{0\le t\le \frac{T}{\varepsilon}}\|u -v^\varepsilon(t)\|_W\le C\eps. 
\end{equation*}
where 
\begin{equation}
\label{eq:rfn}
v^\varepsilon(t,x) = \sum_{j \in \Z^d} a_j(\varepsilon t) e^{i j \cdot x - it|j|^2}. 
\end{equation}
\end{proposition}
\begin{remark}
  Even though it is not emphasized in the notation, the function $u$
  obviously depends on $\eps$, which is present in \eqref{eq:nls}.
\end{remark}
We give below a (complete but short) proof of this result using
geometric optics. Let us mention however that we can also prove this
proposition using a Birkhoff transformation of \eqref{eq:nls} in {\em
  resonant normal form} as in \cite{KP}. We give some details
below. There is also an obvious connection with the {\em modulated
  Fourier expansion} framework developed in \cite{GL} in the
non-resonant case (see also
\cite{vV05}). 

\subsection{Solution of the resonant system}
The first part of Proposition \ref{prop:approx} is a consequence of
the following result:  
\begin{lemma}\label{lem:exista}
  Let $\alpha = (\alpha_j)_{j\in \Z^d}\in \ell^1(\Z^d)$. There exists $T>0$ and
  a unique analytic solution $(a_j)_{j\in \Z^d}:[0,T] \to \ell^1(\Z^d)$ to
  the system \eqref{eq:a}. 
Moreover, there exists constants $M$ and $R$ such that for all $n \in \N$ and
all $s \le T$,  
\begin{equation}
\label{eq:anal}
\forall\, j \in \Z^d,\quad \left|\frac{\dd^n a_j}{\dd t^n}(s) \right|
\le M R^n n!  
\end{equation}
\end{lemma}
\begin{proof}
In \cite{CDS10}, the existence of a time $T_1$ and continuity in time
of the solution $a(t) = (a_j(t))_{j \in \Z^d}$ in $\ell^1$ is
proved. As $\ell^1$ is an algebra, a bootstrap  
 argument shows that $a(t) \in C^\infty\([0,T_1];\ell^1(\Z^d)\)$. 
 From \eqref{eq:a} we immediately obtain for $s \in [0,T_1]$, 
$$
\| \dot a(s) \|_{\ell^1} \le 3\| a(s) \|_{\ell^1}^3,
$$
and by induction
$$
\left\lVert a^{(n)}(s) \right\rVert_{\ell^1} \le 3\cdot 5 \cdots
\cdot (2n + 1) \| a(s) 
\|_{\ell^1}^{2n +1}, 
$$
where $a^{(n)}(t)$ denote the $n$-th derivative of $a(t)$ with respect
to time. This implies
$$
\left\lVert a^{(n)}(0) \right\rVert_{\ell^1} \le 3\cdot 5 \cdots
\cdot (2n + 1) \| \alpha \|_{\ell^1}^{2n +1} \le \| \alpha
\|_{\ell^1} n ! \(3 \| \alpha \|_{\ell^1}^2\)^n, 
$$
which shows the analyticity of $a$ for $t \le T_2 = \frac{1}{6} \|
\alpha \|_{\ell^1}^{-2}$. The estimate \eqref{eq:anal} is then a
standard consequence of Cauchy estimates applied to the complex power
series $\sum_{n \in \N} \frac{1}{n!}a^{(n)}(0) z^n$ defined in the
ball $B(0,2T)$ where $2T = \min(T_1,T_2)$.  
\end{proof}

\subsection{Geometric optics}
\label{sec:scaling}
Let us introduce the scaling
\begin{equation}
  \label{eq:scaling}
  \ttt = \varepsilon t, \quad \tx = \varepsilon x, \quad 
  u(t,x)= \tu^\eps \(\eps t,\eps x\).
\end{equation}
Then \eqref{eq:nls} is \emph{equivalent} to:
\begin{equation}
  \label{eq:nlssemi}
  i\eps\d_\ttt \tu^\eps +\eps^2 \Delta \tu^\eps = \l \eps |\tu^\eps|^2
  \tu^\eps\quad ;\quad \tu^\eps(0,\tx) =
  u_0\(\frac{\tx}{\eps}\)=\sum_{j\in \Z^d}\alpha_j e^{ij\cdot \tx/\eps}.
\end{equation}
In the limit $\eps\to 0$, multiphase geometric optics provides an
approximate solution for \eqref{eq:nlssemi}. The presence of the
factor $\eps$ in front of the nonlinearity has two consequences: in
the asymptotic regime $\eps\to 0$, the eikonal equation is the same as
in the linear case $\l=0$, but the transport equation describing the
evolution of the amplitude is nonlinear. This explains why this
framework is referred to as \emph{weakly nonlinear geometric
  optics}. Note that simplifying by $\eps$ in the Schr\"odinger
equation \eqref{eq:nlssemi}, we can view the limit $\eps\to 0$ as a
small dispersion limit, as in e.g. \cite{Kuksin2}.

We sketch the approach described more precisely in \cite{CDS10}. 
The approximate solution provided by geometric optics has the form
\begin{equation}
\label{eq:veps}
  \tv^\eps(\ttt,\tx) = \sum_{j\in\Z^d}a_j(\ttt) e^{\phi_j(\ttt,\tx)/\eps},
\end{equation}
where we demand $\tu^\eps=\tv^\eps$ at time $\ttt=0$, that is
\begin{equation*}
  a_j(0)=\alpha_j\quad ;\quad \phi_j(0,\tx )=j\cdot \tx. 
\end{equation*}
Plugging this ansatz into \eqref{eq:nlssemi} and ordering the powers
of $\eps$, we find, for the $\O(\eps^0)$ term:
\begin{equation*}
  \d_\ttt \phi_j +|\nabla \phi_j|^2 =0 \quad ;\quad \phi_j(0,\tx)=j\cdot \tx.
\end{equation*}
The solution is given explicitly by
\begin{equation}
  \label{eq:phi}
  \phi_j(\ttt,\tx)= j\cdot \tx - \ttt\lvert j\rvert^2. 
\end{equation}
The amplitude $a_j$ is given by the $\O(\eps^1)$ and is given by
equation \eqref{eq:a} after   projecting the wave along the
oscillation $e^{i\phi_j/\eps}$ according to the  
 the set of \emph{resonant} phases given by $I_j$ (Eqn. \eqref{eq:Ij}). 
By doing so, we have dropped the oscillations of the form $e^{i(k\cdot
  \tx-\om \ttt)/\eps}$ with $\om\not = |k|^2$, generated by nonlinear
interaction: the phase $k\cdot
  \tx-\om \ttt$ does not solve the eikonal equation, and the corresponding
  term is negligible in the limit $\eps\to 0$ thanks to a
  non-stationary phase argument. 

Proposition \ref{prop:approx} is a simple corollary of 
the following result that is established in \cite{CDS10}. We sketch
the proof in
Appendix~\ref{sec:wnlgo}.  
\begin{proposition}\label{prop:approxOG}
  Let $\(\alpha_j\)\in
\ell^1(\Z^d)$, and $\tv^\eps$ be defined by \eqref{eq:veps} and
\eqref{eq:phi}. Then there exists $\eps_0(T)>0$ such that for 
$0<\eps\le \eps_0(T)$, the exact solution to 
\eqref{eq:nlssemi} satisfies $\tu^\eps \in C([0,T];W)$, where $T$ is
given by Lemma~\ref{lem:exista}. In addition,
$\tv^\eps$ approximates $\tu^\eps$ up to $\O(\eps)$: there exists $C$
independent of $\eps\in ]0,\eps_0(T)]$ such that
\begin{equation*}
  \sup_{0\le t\le T}\|\tu^\eps(t) -\tv^\eps(t)\|_W\le C\eps. 
\end{equation*}
\end{proposition}

\subsection{Link with normal forms}
\label{sec:normal}

Viewed as an infinite dimensional Hamiltonian system, \eqref{eq:nls}
can also be interpreted as the equation associated with the
Hamiltonian  
$$
H^\varepsilon(u,\bar u) = H_0 + \varepsilon P := \sum_{j \in \Z^d}
|j|^2 |u_j|^2 +  \varepsilon \frac{ \lambda}{2} \sum_{k + m = j +
  \ell} u_{k} u_m \bar u_\ell \bar u_j.  
$$
that is $i u_j = \partial_{\bar u_j} H(u,\bar u)$, see for instance
the presentations in \cite{Bam03,Greb07} and \cite{FG10}.  

In this setting, the Birkhoff normal form approach consists in searching a
transformation $\tau(u) = u + \mathcal{O}(\varepsilon u^3 )$ close to the
identity over bounded set in the Wiener algebra $W$, and such that in
the new variable $ v = \tau(u)$, the Hamiltonian 
$K^\varepsilon(v,\bar v) = H^\varepsilon(u,\bar u)$ takes the form
$K^\varepsilon = H_0 + \varepsilon Z + \varepsilon^2 R$ where $Z$ is
expected to be as simple as possible and $R =
\mathcal{O}(u^6)$. Searching $\tau$ as the time $t 
= \varepsilon$ flow of an unknown Hamiltonian $\chi$, we are led to
solving the {\em homological} equation  
$$
\{ H_0,\chi\} + Z = P,
$$
where $\{\, \cdot \, , \, \cdot \, \}$ is the Poisson bracket of the
underlying (complex) Hamiltonian structure. Now with unknown
Hamiltonians  
$\chi(u,\bar u) = \sum_{k,m,j,\ell} \chi_{k m \ell j} u_{k} u_m \bar
u_\ell \bar u_j$ and $Z(u,\bar u) = \sum_{k,m,j,\ell} Z_{km\ell j}
u_{k} u_m \bar u_\ell \bar u_j$,  
the previous relation can be written
$$
(|k|^2 + |m|^2 - |j|^2 - |\ell|^2)\chi_{k m \ell j} + Z_{k m \ell j} =
P_{k m \ell j},  
$$
where
$$
P_{km\ell j} = \left\{\begin{array}{rl}
1 &  \mbox{if} \quad k + m - j - \ell = 0,\\[2ex]
0 &   \mbox{otherwise.}
\end{array}\right.
$$
The solvability of this equation relies precisely on the resonant
relation $|k|^2 + |m|^2 = |j|^2 + |\ell|^2$: For non resonant indices, we
can solve for $\chi_{k m \ell j}$ and set $Z_{k m \ell j} = 0$, while
for resonant indices, we must take $Z_{km\ell j} = P_{km \ell
  j}$. Note that here there is no small divisors issues, as the
denominator is always  an integer ($0$ or greater than $1$).  

Hence we see that up to $\mathcal{O}(\epsilon^2)$ terms as long as the
solution remains bounded in $W$, the dynamics 
in the new variable will be close to the dynamics associated with the
Hamiltonian  
$$
K_1^\varepsilon(u,\bar u) = H_0 + \varepsilon Z := \sum_{j \in \Z^d}
|j|^2 |u_j|^2 +  \varepsilon \frac{\lambda}{2} \sum_{\substack{k + m =
    j + \ell \\ |k|^2 + |m|^2 = |j|^2 + |\ell|^2}} u_{k} u_m \bar
u_\ell \bar u_j.  
$$
At this point, let us observe that $H_0$ and $Z$ commute: $\{H_0,Z\} =
0$, and hence the dynamics of $K_1^\varepsilon$ is the simple
superposition of the dynamics of $H_0$ (the phase oscillation
\eqref{eq:phi}) to the dynamics of $\varepsilon Z$ (the resonant
system \eqref{eq:a}). Hence we easily calculate that
$v^\varepsilon(t,x)$ defined in \eqref{eq:rfn} is the exact solution
of the Hamiltonian $K_1^\varepsilon$. In other words, it is the
solution of the {\em first resonant normal form} of the system
\eqref{eq:nls}.  

The approximation result can then easily be proved using estimates  on
the remainder terms (that can be controlled in the Wiener
algebra, see \cite{FG10}),  in combination with Lemma
\ref{lem:exista} which ensures the stability of the solution of
$K_1^\varepsilon$ and a uniform bound in the Wiener algebra over a
time of order $1/\varepsilon$.

\section{An iterative approach}
\label{sec:interative}

We now turn to the analysis of the resonant system \eqref{eq:a}. 
The main remark for the forthcoming analysis is that new modes can be
generated by nonlinear interaction: we may have $a_j\not =0$ even
though $\alpha_j=0$. We shall view this phenomenon from a dynamical
point of view. As a first step, we recall the description of the sets
of resonant phases, established in \cite{Iturbulent} in the case $d=2$
(the argument remains the same for $d\ge 2$, see \cite{CDS10}):
\begin{lemma}\label{lem:completion}
  Let $j\in \Z^d$. Then, $(k,\ell,m) \in I_j$ 
precisely when the endpoints of the vectors $k, \ell, m,j$ form four
corners of a non-degenerate rectangle 
with $\ell$ and $j$ opposing each other, or when this
quadruplet corresponds to one of the two following degenerate cases:
$(k=j, m=\ell)$,  or $(k=\ell, m=j)$. 
\end{lemma}
  As a matter of fact, (the second part of) this lemma remains true in
  the one-dimensional case $d=1$. A specifity of that case, though, is
  that the associated transport equations show that no mode can
  actually be created \cite{CDS10}. The reason is that 
Lemma~\ref{lem:completion} implies that when
  $d=1$, \eqref{eq:a} takes the form $i\dot a_j =
  M_j a_j$ for some (smooth and real-valued) function $M_j$ whose
  exact value is unimportant: if $a_j(0)=0$, then $a_j(t)=0$ for all
  $t$. In the present paper, on the contrary, we examine precisely the
  appearance  of new modes.  
\smallbreak

Introduce the set of initial modes:
\begin{equation*}
  J_0=\{j\in \Z^d\mid \ \alpha_j\not =0\}. 
\end{equation*}
In view of \eqref{eq:a}, modes which appear after one iteration of
Lemma~\ref{lem:completion} are given by:
\begin{equation*}
  J_1=\{j\in \Z^d\setminus J_0 \mid \ \dot a_j(0)\not =0\}. 
\end{equation*}
One may also think of $J_1$ in terms of Picard iteration. Plugging the
initial modes (from $J_0$) into the nonlinear Duhamel's term and
passing to the limit $\eps\to 0$, $J_1$ corresponds to the new modes
resulting from this manipulation. 
More generally, modes appearing after $k$ iterations exactly are
characterized by: 
\begin{equation*}
  J_k=\left\{j\in \Z^d\setminus \bigcup_{\ell=0}^{j-1}J_\ell \mid \
  \frac{\dd^k}{\dd t^k} a_j(0)\not =0\right\}.  
\end{equation*}

\section{A particular dynamical system}
\label{sec:dyn}

We consider the initial datum 
\begin{equation}
  \label{eq:bonneci}
  u_0(x) = 1+2\cos x_1+2\cos x_2=1 + e^{ix_1}+e^{-ix_1}+e^{ix_2}+e^{-ix_2}.
\end{equation}
The corresponding set of initial modes is given by
\begin{equation*}
  J_0=\left\{ (0,0), (1,0),
    (-1,0),(0,1),(0,-1)\right\}\times\{0_{\Z^{d-2}}\}.  
\end{equation*}
It is represented on the following figure:
\setlength{\unitlength}{1cm}
\begin{center}
\begin{picture}(4,4)
\linethickness{0.1mm}
\put(0.3,1){\line(1,0){3.4}}
\put(0.3,0.5){\line(1,0){3.4}}
\put(0.3,1.5){\line(1,0){3.4}}
\put(0.3,2.5){\line(1,0){3.4}}
\put(0.3,3.5){\line(1,0){3.4}}
\put(0.3,3){\line(1,0){3.4}}
\put(0.5,0.3){\line(0,1){3.4}}
\put(1,0.3){\line(0,1){3.4}}
\put(1.5,0.3){\line(0,1){3.4}}
\put(2.5,0.3){\line(0,1){3.4}}
\put(3,0.3){\line(0,1){3.4}}
\put(3.5,0.3){\line(0,1){3.4}}
\linethickness{0.4mm}
\put(0,2){\vector(1,0){4}}
\put(2,0){\vector(0,1){4}}
\put(2,2){\circle*{.25}}
\put(2,1.5){\circle*{.25}}
\put(2.5,2){\circle*{.25}}
\put(2,2.5){\circle*{.25}}
\put(1.5,2){\circle*{.25}}
\end{picture}
\end{center}
In view of Lemma~\ref{lem:completion}, the generation of modes affects
only the first two coordinates: the dynamical system that we study is
two-dimensional, and we choose to drop out the last $d-2$ coordinates
in the sequel, implicitly equal to $0_{\Z^{d-2}}$.
After one iteration of Lemma~\ref{lem:completion}, four points appear:
\begin{equation*}
  J_1=\{(1,1),(1,-1),(-1,-1),(-1,1)\},
\end{equation*}
as plotted below. 
\begin{center}
\begin{picture}(4,4)
\linethickness{0.1mm}
\put(0.3,1){\line(1,0){3.4}}
\put(0.3,0.5){\line(1,0){3.4}}
\put(0.3,1.5){\line(1,0){3.4}}
\put(0.3,2.5){\line(1,0){3.4}}
\put(0.3,3.5){\line(1,0){3.4}}
\put(0.3,3){\line(1,0){3.4}}
\put(0.5,0.3){\line(0,1){3.4}}
\put(1,0.3){\line(0,1){3.4}}
\put(1.5,0.3){\line(0,1){3.4}}
\put(2.5,0.3){\line(0,1){3.4}}
\put(3,0.3){\line(0,1){3.4}}
\put(3.5,0.3){\line(0,1){3.4}}
\linethickness{0.4mm}
\put(0,2){\vector(1,0){4}}
\put(2,0){\vector(0,1){4}}
\put(2,2){\circle*{.25}}
\put(2,1.5){\circle*{.25}}
\put(2.5,2){\circle*{.25}}
\put(2,2.5){\circle*{.25}}
\put(1.5,2){\circle*{.25}}
\color{blue}
\put(1.5,1.5){\circle*{.25}}
\put(2.5,1.5){\circle*{.25}}
\put(1.5,2.5){\circle*{.25}}
\put(2.5,2.5){\circle*{.25}}
\end{picture}
\end{center}
The next two steps are described geometrically: 
\begin{center}
\begin{picture}(4,4)
\linethickness{0.1mm}
\put(0.3,1){\line(1,0){3.4}}
\put(0.3,0.5){\line(1,0){3.4}}
\put(0.3,1.5){\line(1,0){3.4}}
\put(0.3,2.5){\line(1,0){3.4}}
\put(0.3,3.5){\line(1,0){3.4}}
\put(0.3,3){\line(1,0){3.4}}
\put(0.5,0.3){\line(0,1){3.4}}
\put(1,0.3){\line(0,1){3.4}}
\put(1.5,0.3){\line(0,1){3.4}}
\put(2.5,0.3){\line(0,1){3.4}}
\put(3,0.3){\line(0,1){3.4}}
\put(3.5,0.3){\line(0,1){3.4}}
\linethickness{0.4mm}
\put(0,2){\vector(1,0){4}}
\put(2,0){\vector(0,1){4}}
\put(2,2){\circle*{.25}}
\put(2,1.5){\circle*{.25}}
\put(2.5,2){\circle*{.25}}
\put(2,2.5){\circle*{.25}}
\put(1.5,2){\circle*{.25}}
\color{blue}
\put(1.5,1.5){\circle*{.25}}
\put(2.5,1.5){\circle*{.25}}
\put(1.5,2.5){\circle*{.25}}
\put(2.5,2.5){\circle*{.25}}
\color{red}
\put(2,1){\circle*{.25}}
\put(2,3){\circle*{.25}}
\put(1,2){\circle*{.25}}
\put(3,2){\circle*{.25}}
\end{picture}
\hspace{1cm}
\begin{picture}(4,4)
\linethickness{0.1mm}
\put(0.3,1){\line(1,0){3.4}}
\put(0.3,0.5){\line(1,0){3.4}}
\put(0.3,1.5){\line(1,0){3.4}}
\put(0.3,2.5){\line(1,0){3.4}}
\put(0.3,3.5){\line(1,0){3.4}}
\put(0.3,3){\line(1,0){3.4}}
\put(0.5,0.3){\line(0,1){3.4}}
\put(1,0.3){\line(0,1){3.4}}
\put(1.5,0.3){\line(0,1){3.4}}
\put(2.5,0.3){\line(0,1){3.4}}
\put(3,0.3){\line(0,1){3.4}}
\put(3.5,0.3){\line(0,1){3.4}}
\linethickness{0.4mm}
\put(0,2){\vector(1,0){4}}
\put(2,0){\vector(0,1){4}}
\put(2,2){\circle*{.25}}
\put(2,1.5){\circle*{.25}}
\put(2.5,2){\circle*{.25}}
\put(2,2.5){\circle*{.25}}
\put(1.5,2){\circle*{.25}}
\color{blue}
\put(1.5,1.5){\circle*{.25}}
\put(2.5,1.5){\circle*{.25}}
\put(1.5,2.5){\circle*{.25}}
\put(2.5,2.5){\circle*{.25}}
\color{red}
\put(2,1){\circle*{.25}}
\put(2,3){\circle*{.25}}
\put(1,2){\circle*{.25}}
\put(3,2){\circle*{.25}}
\color{green}
\put(1,1){\circle*{.25}}
\put(1,1.5){\circle*{.25}}
\put(1,2.5){\circle*{.25}}
\put(1,3){\circle*{.25}}
\put(1.5,1){\circle*{.25}}
\put(1.5,3){\circle*{.25}}
\put(3,3){\circle*{.25}}
\put(3,2.5){\circle*{.25}}
\put(3,1.5){\circle*{.25}}
\put(3,1){\circle*{.25}}
\put(2.5,1){\circle*{.25}}
\put(2.5,3){\circle*{.25}}
\end{picture}
\end{center}
As suggested by these illustrations, we can prove by induction:
\begin{lemma}
Let $p\in \N$. 
\begin{itemize}
\item The set of relevant modes after $2p$ iterations is the square
  of length $2^{p}$ whose diagonals are parallel to the axes:
  \begin{equation*}
   \Nc^{(2p)}:= \bigcup_{\ell=0}^{2p}J_\ell = \left\{(j_1,j_2)\mid
      |j_1|+|j_2|\le 2^p\right\}. 
  \end{equation*}
\item The set of relevant modes after $2p+1$ iterations is the square
  of length $2^{p+1}$ whose sides are parallel to the axes:
  \begin{equation*}
   \Nc^{(2p+1)} := \bigcup_{\ell=0}^{2p+1}J_\ell = \left\{(j_1,j_2)\mid
      \max(|j_1|,|j_2|)\le 2^p\right\}. 
  \end{equation*}
\end{itemize}
 \end{lemma}
After an infinite number of iterations, the whole lattice $\Z^2$ is
generated:
\begin{equation*}
  \bigcup_{k\ge 0}\Nc^{(k)}=\Z^2 \times \{0_{\Z^{d-2}}\}. 
\end{equation*}
Among these sets, our interest will focus on \emph{extremal modes}:
for $p\in \N$,
\begin{align*}
 \Nc_*^{(2p)} &:= \left\{(j_1,j_2)\in \{(0,\pm 2^p),(\pm
    2^p,0)\}\right\},\\
  \Nc_*^{(2p+1)}&:= \left\{(j_1,j_2)\in \{(\pm 2^p,\pm 2^p),(\mp
    2^p,\pm 2^p)\}\right\}.
\end{align*}
These sets correspond to the edges of the squares obtained successively
by iteration of Lemma~\ref{lem:completion} on $J_0$. 
The set $\Nc_*$
defined in Theorem~\ref{theo:main} corresponds to 
\begin{equation*}
  \Nc_* = \bigcup_{k\ge 0}\Nc_*^{(k)}. 
\end{equation*}
The important property associated to these extremal points is that
they are generated in a unique fashion:
\begin{lemma}\label{lem:creation}
  Let $n\ge 1$, and $j\in \Nc_*^{(n)}$. There exists a unique pair
  $(k,m)\in \Nc^{(n-1)}\times \Nc^{(n-1)}$ such that $j$ is generated
  by the interaction of the modes $0$, $k$ and $m$, up to the permutation of
  $k$ and $m$. More precisely, $k$
  and $m$ are extremal points generated at the previous step:
  $k,m\in\Nc^{(n-1)}_*$.  
\end{lemma}
Note however that points in $\Nc_*^{(n)}$ are generated in a
non-unique fashion by the interaction of modes in $\Z^d$. For
instance, $(1,1)\in J_1$ is generated after one step only by the
interaction of $(0,0)$, $(1,0)$ and $(0,1)$. On the other hand, we see
that after two iterations, $(1,1)$ is fed also by the interaction of
the other three points in $\Nc_*^{(1)}$, $(-1,1)$, $(-1,-1)$ and
$(1,-1)$. After three iterations, there are even more three waves
interactions affecting $(1,1)$.

\begin{remark}
According to 
the numerical experiment performed in the last section, it seems that
{\em all} modes --- and not only the extremal ones ---  receive some energy in
the time interval  
$[0,T/\eps]$. However the dynamics for the other modes
is much more complicated to understand, as non extremal points of
$\Nc^{(n+1)}$ are in general generated by several triplets of points in
$\Nc^{(n)}$.  
\end{remark}

\section{Proof of Theorem~\ref{theo:main}}
\label{sec:proof}

Since the first part of Theorem~\ref{theo:main} has been established
at the end of \S\ref{sec:og}, we now focus our attention on the estimates
announced in  Theorem~\ref{theo:main}. 

In view of the geometric analysis of the previous section, we will show that can
compute the first non-zero term in the Taylor expansion of solution
$a_m(t)$ of \eqref{eq:a} at 
$t=0$, for $m\in \Nc_*$. Let $n\ge 1$ and $j\in \Nc_*^{(n)}$. 
Note that since we have considered initial coefficients which are all
equal to one --- see 
\eqref{eq:bonneci} --- and because of the symmetry in \eqref{eq:a},
the coefficients $a_j(t)$ do not depend on $j \in \Nc_*^{(n)}$ but
only on $n$.

Hence we have
\begin{equation*}
  a_j(t) =
  \frac{t^{\alpha(n)}}{\alpha(n)!}\frac{\dd^{\alpha(n)}a_j}{\dd t^{\alpha(n)}}(0)
  + 
  \frac{t^{\alpha(n)+1}}{\alpha(n)!}\int_0^1 (1-\theta)^{\alpha(n)}
  \frac{\dd^{\alpha(n)+1}a_j}{\dd t^{\alpha(n)+1}}(\theta t)d\theta, 
\end{equation*}
for some $\alpha(n)\in \N$ still to be determined.

First, Eqn.~\eqref{eq:anal} in 
Lemma~\ref{lem:exista} ensures that there exists $C_0>0$
\emph{independent of $j$ and $n$} such that 
\begin{equation*}
  r_j(t) = \frac{t^{\alpha(n)+1}}{\alpha(n)!}\int_0^1 (1-\theta)^{\alpha(n)}
  \frac{d^{\alpha(n)+1}a_j}{dt^{\alpha(n)+1}}(\theta t)d\theta
\end{equation*}
satisfies
\begin{equation}\label{eq:rj}
  |r_j(t)|\le (C_0t)^{\alpha(n)+1}.
\end{equation}
Next, we write
\begin{equation}\label{eq:aj-approx}
  a_j(t) = c(n) t^{\alpha(n)} +r_j(t),
\end{equation}
and we determine $c(n)$ and $\alpha(n)$ thanks to the iterative approach
analyzed in the previous paragraph.  In view of
Lemma~\ref{lem:creation}, we have
\begin{equation*}
  i\dot a_j = 2\l c(n-1)^2 t^{2\alpha(n-1)} +
  \O\(t^{2\alpha(n-1)+1}\), 
\end{equation*}
where the factor $2$ accounts for the fact that the vectors $k$ and
$m$ can be exchanged in Lemma~\ref{lem:creation}. We infer the
relations:
\begin{align*}
  \alpha(n) &= 2\alpha(n-1)+1 \quad ;\quad \alpha(0)=0.\\
c(n)&= -2i\l\frac{c(n-1)^2}{2\alpha(n-1)+1}\quad ;\quad c(0)=1. 
\end{align*}
We first derive
\begin{equation*}
  \alpha(n) = 2^n-1.
\end{equation*}
We can then compute, $c(1)=-2i\l$, and for $n\ge 1$:
\begin{equation*}
  c(n+1)= i\frac{(2\l)^{\sum_{k=0}^n
    2^k}}{\prod_{k=1}^{n+1}\(2^k-1\)^{2^{n+1-k}}}=
i\frac{(2\l)^{2^{n+1}-1}}{\prod_{k=1}^{n+1}\(2^k-1\)^{2^{n+1-k}}} .  
\end{equation*}
We can then infer the first estimate of Theorem~\ref{theo:main}: by
Proposition~\ref{prop:approxOG}, there exists $C$ independent of $j$
and $\eps$ such that for $0<\eps\le \eps_0(T)$, 
\begin{equation*}
  \left\lvert u_j(t)-a_j(\eps t)\right\rvert \le C\eps,\quad 0\le t\le
  \frac{T}{\eps}.
\end{equation*}
We notice that since for $j\in \Nc_*^{(n)}$, $|j| = 2^{n/2}$, regardless of the
parity of $n$, we have $\alpha(n) = |j|^2 - 1$. 
For $j\in \Nc_*$, we then use \eqref{eq:aj-approx} and \eqref{eq:rj},
and the estimate follows, with $c_j=c(n)$. 
\smallbreak

To prove the last estimate of Theorem~\ref{theo:main}, we must examine
more closely the behavior of $c(n)$. 
Since $2^k-1\le 2^k$ for $k\ge 1$, we have the estimate
\begin{equation*}
  |c(n+1)|\ge \frac{2^{2^{n+1}-1}}{2^{\sum_{k=1}^{n+1}k2^{n+1-k}}}.
\end{equation*}
Introducing the function
\begin{equation*}
  f_{n+1}(x) = \sum_{k=1}^{n+1}x^k= \frac{1-x^{n+1}}{1-x}x,\quad x\not
  =1,
\end{equation*}
we have
\begin{equation*}
  \sum_{k=1}^{n+1}k2^{n+1-k} = 2^n f_{n+1}'\(\frac{1}{2}\) = 2^{n+2}-n-3,
\end{equation*}
and the (rough) bound
\begin{equation*}
  |c(n+1)| \ge 2^{2^{n+1}-1-2^{n+2}+n+3}= 2^{-2^{n+1}+n+2}\ge
  2^{-2^{n+1}}. 
\end{equation*}
We can now gather all the estimates together: 
\begin{align}
\label{eq:final}
  |u_j(t)|&\ge \left\lvert c(n) \(\eps t\)^{\alpha(n)}\right\rvert -
  \(C_0\eps t\)^{\alpha(n)+1} - C\eps \nonumber\\
&\ge \frac{1}{2}\(\frac{\eps
    t}{2}\)^{2^n-1} -  \(C_0\eps t\)^{2^n}- C\eps \nonumber\\
&\ge \frac{1}{2}\(\frac{\eps
    t}{2}\)^{2^n-1}\(1 - (2 C_0)^{2^n}\eps t\)-C\eps. 
\end{align}
To conclude, we simply consider $t$ such that
\begin{equation}
\label{eq:time}
  \(\frac{\eps
    t}{2}\)^{2^n-1}= \eps^\g,\text{ that is }t =
  \frac{2}{\eps^{1-\g/\alpha(n)}}. 
\end{equation}
Hence for the time $t$
given in \eqref{eq:time}, since $\alpha(n) = |j|^2 - 1$, we have  
\begin{align*}
(2 C_0)^{2^n}\eps t
  &= \(2 C_0\)^{|j|^2} \eps^{\g/(|j|^2 - 1)}\\
    &= \exp \(|j|^2 \log (2C_0)-\frac{\g}{|j|^2 - 1}\log
  \(\frac{1}{\eps}\)\).
  \end{align*}
Assuming the spectral
localization
\begin{equation*}
  |j|\le \alpha \(\log \frac{1}{\eps}\)^\theta,
\end{equation*}
we get for $\eps$ small enough
\begin{equation*}
  (2 C_0)^{2^n}\eps t \le \exp\(\alpha^2 \(\log
  \frac{1}{\eps}\)^{2\theta}\log (2C_0) -  
\frac{\g}{\alpha^2} \(\log\frac{1}{\eps}\)^{1 - 2\theta}\). 
\end{equation*}
The argument of the exponential goes to $-\infty$ as $\eps \to 0$
provided that 
\begin{equation*}
 \g > 0 \quad \mbox{and}\quad  \theta<\frac{1}{4},  
\end{equation*}
in which case we have $1 - (2 C_0)^{2^n}\eps t > 3/4$ for $\eps$
sufficiently small. Inequality~\eqref{eq:final} then yields the result,
owing to the fact that  $C\eps$ is negligible compared to $\eps^\g$
when $0\le \g< 1$. 
\smallbreak

Finally, we note that the choice \eqref{eq:time} is consistent with
$\eps t\in [0,T]$ for $\eps\le \eps_0$, for some $\eps_0>0$ uniform
in $j$ satisfying the above spectral localization, since
\begin{equation*}
  \eps^{\g/\alpha(n)}=e^{-\frac{\g}{\alpha(n)}\log\frac{1}{\eps}}\le
  \exp\( -\frac{\g}{\alpha^2}\(\log\frac{1}{\eps}\)^{1-2\theta}\)\Tend
  \eps 0 0 .
\end{equation*}


\section{Numerical illustration}

We consider the equation \eqref{eq:nls0} in the defocusing case
($\lambda = 1$) on the two-dimensional torus, $d=2$. We take $u(0,x) = \delta
( 1 + 2 \cos( x_1) + 2 \cos(x_2))$ with $\delta = 0.0158$. With the
previous notations, this corresponds to $1/\eps = \delta^{-2} \simeq
4.10^{3}$.  In Figure ~\ref{fig:1},
we plot the evolution of the logarithms of the Fourier modes $\log
|u_j(t)|$ for $ j = (0,n)$, with $n = 0,\ldots, 15$. We observe the
energy exchanges between the modes. Note that all the modes (and not
only the extremal modes in the set $\Nc_*$) gain some energy, but that
after some time there is a stabilization effect (all the modes are
turned on) and the energy exchanges are less significant.  

\begin{figure}[ht]
\begin{center}
\rotatebox{0}{\resizebox{!}{0.4\linewidth}{%
   \includegraphics{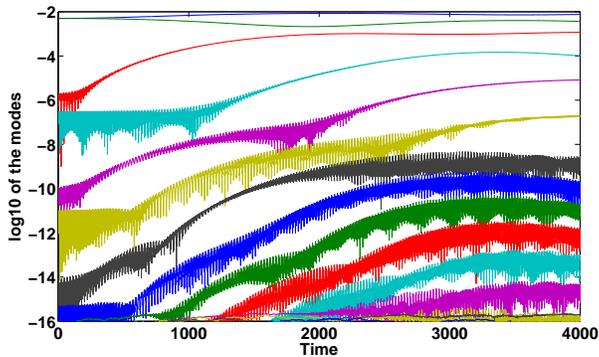}}}
\end{center}
\caption{Evolution of the Fourier modes of the resonant solution in
  logarithmic scale: $u_0(x) = 1 + 2 \cos( x_1) + 2 \cos(x_2)$ } 
\label{fig:1}
\end{figure}

In contrast, we plot in Figure~\ref{fig:2} the solution corresponding to the
initial value $u(0,x) = \delta (2 \cos( x_1) + 2 \cos(x_2))$ with the
same $\delta$. In this situation, no energy exchanges are observed
after a relatively long time. Note that in this case, the initial data
is made of the $4$ modes $\{j \in \Z^2\, | |j| = 1\,\}$ forming a
square in $\Z^2$. This set is closed for the resonance relation, so  no
energy exchange is expected in the time scale
$\mathcal{O}(1/\eps)$.  We notice that 
the solution of \eqref{eq:a} is given explicitly by $a_j(t) = 0$ for
 $|j| \neq 1$, and $a_j(t) = \exp(9 i t)$ for $|j| = 1$. 

\begin{figure}[ht]
\begin{center}
\rotatebox{0}{\resizebox{!}{0.4\linewidth}{%
   \includegraphics{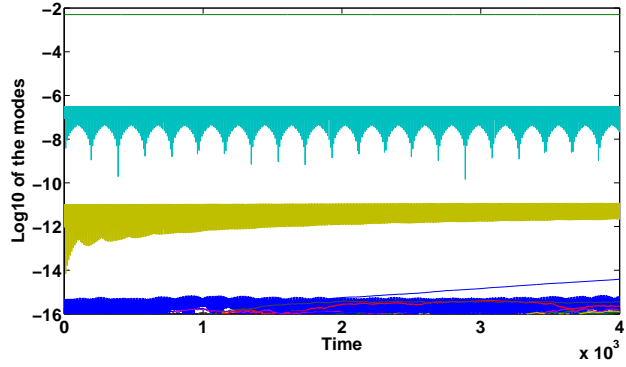}}}
\end{center}
\caption{Evolution of the Fourier modes of the nonresonant solution in
  logarithmic scale: $u_0(x) = 2 \cos( x_1) + 2 \cos(x_2)$} 
\label{fig:2}
\end{figure}

The numerical scheme is a splitting time integrator based on the
decomposition between the Laplace operator and the nonlinearity in
combination with a Fourier pseudo-spectral collocation method (see for
instance \cite{Lub08} and \cite[Chap IV]{FaouBook} for convergence
results in the case of \eqref{eq:nls0}). While the Laplace operator
part $i \partial_t u = - \Delta u$ can be integrated exactly in
Fourier, the solution of the nonlinear part  $i \partial_t u = |u|^2
u$ starting in $v(x)$ is given explicitly by the formula $u(t,x) =
\exp(-i t |v(x)|^2) v(x)$. The fast Fourier transform algorithm allows
an easy implementation of the algorithm. The stepsize used is $\tau =
0.001$ and a $128 \times 128$ grid is used.

Note that using the framework of  \cite{FG11,FaouBook}, we can prove
that the numerical solution can be interpreted as the exact solution
of a modified Hamiltonian of the form  
$$
\sum_{j \in B_K}
|j|^2 |u_j|^2 +  \frac{ \lambda}{2} \sum_{\substack{(k,m,j,\ell) \in
    B_K \\ k + m - j - 
  \ell \in K \Z^2}} \frac{i \tau \omega_{km \ell j}}{\exp(i \tau
\omega_{km\ell j}) - 1}u_{k} u_m \bar u_\ell \bar u_j +
\mathcal{O}(\tau), 
$$
where $\omega_{km \ell j} = |k|^2 + |m|^2 - |\ell|^2 - |m|^2$ and
$B_K$ the grid of frequencies $j = (j_1,j_2)$ such that $j_1$ and
$j_2$ are less than $K/2 = 64$. Note that this energy is well defined
as $\tau \omega_{km\ell j}$ is never a multiple of $2\pi$, and that
the frequencies of the linear operator of this modified energy carry
on the same resonance relations (at least for relatively low
modes). This partly explains why the cascade effect due to the
resonant system should be correctly reproduced by the numerical
simulations. 
\appendix

\section{Sketch of the proof of Proposition~\ref{prop:approxOG}}
\label{sec:wnlgo}

By construction, the approximate solution $\tv^\eps$ solves
\begin{equation*}
  i\eps \d_t \tv^\eps +\eps^2 \Delta \tv^\eps =\l \eps
  |\tv^\eps|^2\tv^\eps +\l\eps r^\eps,
\end{equation*}
where the source term $r^\eps$ correspond to non-resonant interaction
terms which have been discarded:
\begin{equation*}
  r^\eps (t,x) =\sum_{j\in \Z^d}\sum_{(k,\ell,m)\not\in I_j}a_k(t)
  \overline a_\ell(t) a_m(t)
  e^{i(\phi_k(t,x)-\phi_\ell(t,x)+\phi_m(t,x))/\eps}. 
\end{equation*}
We write
\begin{equation*}
 \phi_k(t,x)-\phi_\ell(t,x)+\phi_m(t,x) = \kappa_{ k,\ell,m} \cdot x -
 \omega_{k,\ell,m}t,
\end{equation*}
with $\kappa_{ k,\ell,m}\in \Z^d$, $\omega_{k,\ell,m}\in \Z$ and 
$\lvert \kappa_{ k,\ell,m}\rvert^2\not = \omega_{k,\ell,m}$, hence
\begin{equation}\label{eq:gap}
 \left\lvert \lvert \kappa_{
   k,\ell,m}\rvert^2 -\omega_{k,\ell,m} \right\rvert \ge 1.
\end{equation}
The error term ${\tt w}^\eps= \tu^\eps -\tv^\eps$ solves
\begin{equation*}
  i\eps \d_t {\tt w}^\eps +\eps^2 \Delta {\tt w}^\eps =\l \eps
  \( |{\tt w}^\eps+\tv^\eps|^2\({\tt
    w}^\eps+\tv^\eps\)-|\tv^\eps|^2\tv^\eps\) -\l\eps r^\eps\quad
  ;\quad {\tt w}^\eps_{\mid t=0}=0. 
\end{equation*}
By Duhamel's principle, this can be recasted as 
\begin{align*}
  {\tt w}^\eps(t) & = -i\l \int_0^t e^{i\eps(t-s)\Delta}\(|{\tt
    w}^\eps+\tv^\eps|^2\({\tt
    w}^\eps+\tv^\eps\)-|\tv^\eps|^2\tv^\eps\)(s)ds \\
&\quad +i\l \int_0^t  e^{i\eps (t-s)\Delta} r^\eps(s)ds. 
\end{align*}
Denote 
\begin{equation*}
  R^\eps(t)=\int_0^t  e^{i\eps(t-s)\Delta} r^\eps(s)ds.
\end{equation*}
Since $W$ is an algebra, and the norm in $W$ controls the
$L^\infty$-norm, it suffices to prove
\begin{equation*}
  \|R^\eps\|_{L^\infty([0,T];W)}=\O(\eps). 
\end{equation*}
We compute
\begin{equation*}
R^\eps(t,x)=  \sum_{j\in \Z^d}\sum_{(k,\ell,m)\not\in
  I_j}b_{k,\ell,m}(t,x),
  \end{equation*}
where
\begin{equation*}
b_{k,\ell,m}(t,x)=\int_0^t a_k(s)
  \overline a_\ell(s) a_m(s) \exp\( i\frac{\kappa_{ k,\ell,m} \cdot x
    + \lvert \kappa_{ k,\ell,m}\rvert^2s-
 \omega_{k,\ell,m}s}{\eps}\)ds.  
\end{equation*}
Proposition~\ref{prop:approxOG} then follows from one integration by
parts (integrate the exponential), along with \eqref{eq:gap} and
Lemma~\ref{lem:exista}.  

\subsubsection*{Acknowledgements} The authors wish to thank
Beno\^{\i}t Gr\'ebert for stimulating discussions on this work.

\bibliographystyle{amsplain}
\bibliography{../biblio}

\end{document}